\documentclass[12pt]{amsart}
\usepackage{xcolor}
\usepackage{hyperref}
\usepackage{amsmath,amssymb,amsthm,amscd}
\usepackage[shortlabels]{enumitem}
\newtheorem{theorem}{Theorem}
\newtheorem{lemma}[theorem]{Lemma}

\newtheorem{proposition}[theorem]{Proposition}
\newtheorem*{theorem*}{Theorem}

\theoremstyle{definition}
\newtheorem{remark}{Remark}

\usepackage[margin=3cm]{geometry}
\newcommand{\Z}{\mathbb{Z}}

\newcommand{\F}{\mathbb{F}}
\newcommand{\Q}{\mathbb{Q}}
\newcommand{\Gal}{\textrm{Gal}}
\newcommand{\Aut}{\textrm{Aut}}
\newcommand{\oo}{\mathcal{O}}
\newcommand{\Cl}{\mathfrak{L}}
\newcommand{\ssm}{\smallsetminus}
\newcommand{\GL}{\textrm{GL}}
\newcommand{\SL}{\textrm{SL}}

\newcommand{\tors}{\textrm{tors}}
\newcommand{\Stab}{\textrm{Stab}}
\newcommand{\etors}{\emph{tors}}

\newcommand{\cD}{\mathcal{D}}

\newcommand{\oQ}{\overline{\Q}}
\newcommand{\oQl}{\overline{\Q_\ell}}
\newcommand{\odd}{\textrm{odd}}
\newcommand{\lcm}{\textrm{lcm}}
\usepackage{mathrsfs}
\usepackage{url}
\date{}
\author{Tyler Genao}
\thanks{Email: \texttt{tylergenao@uga.edu}. 
This material is based upon work supported by the National Science Foundation Graduate Research Fellowship under Grant No. 1842396. Partial support was also provided by the Research and Training Group grant DMS-1344994 funded by the National Science Foundation.}
\title[Growth of Torsion Groups Upon Base Change From Number Fields]{Growth of Torsion Groups of Elliptic Curves Upon Base Change From Number Fields}
\begin{document}
\begin{abstract}
Given a number field $F_0$ that contains no Hilbert class field of any imaginary quadratic field, we show that under GRH there exists an effectively computable constant $B:=B(F_0)\in\Z^+$ for which the following holds: for any finite extension $L/F_0$ whose degree $[L:F_0]$ is coprime to $B$, one has for all elliptic curves $E_{/F_0}$ that the $L$-rational torsion subgroup  $E(L)[\tors]=E(F_0)[\tors]$. This generalizes a previous result of Gonz\'{a}lez-Jim\'{e}nez and Najman \cite[Theorem 7.2.i]{GJN20} over $F_0=\Q$. 

Towards showing this, we also prove a result on relative uniform divisibility of the index of a mod-$\ell$ Galois representation of an elliptic curve over $F_0$. Additionally, we show that the main result's conclusion fails when we allow $F_0$ to have rationally defined CM, due to the existence of $F_0$-rational isogenies of arbitrarily large prime degrees satisfying certain congruency conditions.
\end{abstract}
\maketitle
%\tableofcontents
\section{Introduction}
Let $E_{/F_0}$ be an elliptic curve defined over a number field $F_0$. By the Mordell-Weil Theorem, the set $E(F_0)$ of $F_0$-rational points on $E$ is a finitely generated abelian group; consequently, the $F_0$-rational torsion subgroup $E(F_0)[\tors]$ is finite. For any extension $L/F_0$, one has the inclusion of subgroups $E(F_0)\subseteq E(L)$. A natural question to ask is how much the torsion subgroup $E(F_0)[\tors]$ grows once base-changed to $L$.

When the base field is $F_0:=\mathbb{Q}$, there are several results which establish constraints on torsion growth uniformly in the degree $[L:\mathbb{Q}]$. For example, when $[L:\Q]<\infty$ Lozano-Robledo gave a linear bound on the prime divisors of $\#E(L)[\tors]$ in terms of $[L:\Q]$ \cite[Theorem 1.3]{LR13}. He did this in part by determining lower bounds on the degrees of fields of definition of torsion points on $E$ of prime order $\ell$, via analyzing the image of the mod-$\ell$ Galois representation of $E$ as a subgroup of $\GL_2(\Z/\ell\Z)$ \`a la Serre \cite{Ser72}.

Continuing this idea of studying Galois representations of elliptic curves to better understand their torsion point degrees, Gonz\'{a}lez-Jim\'{e}nez and Najman \cite{GJN20} have determined restrictions on the degrees of number fields over which the torsion subgroup can grow for rational elliptic curves under base change. One of their results is as follows.
\begin{theorem*}\cite[Theorem 7.2.i]{GJN20}\label{baseChangeOverQ}
For all finite extensions $L/\Q$ whose degree is coprime to $2\cdot 3\cdot 5\cdot 7$, one has for all elliptic curves $E$ defined over $\Q$ that
\[
E(L)[\etors]=E(\Q)[\etors].
\]
\end{theorem*}
Our present paper proves a number field analogue of \cite[Theorem 7.2.i]{GJN20}, wherein our base field can be larger than $\Q$. 
The proof of our theorem is split into Propositions \ref{TorsionDoesntChange} and \ref{AFFinite}.
By GRH we mean the \textit{Generalized Riemann Hypothesis.} Also, given a number field $F_0$, we say that $F_0$ has \textit{rationally defined CM} if there  exists an elliptic curve $E_{/F_0}$ with complex multiplication (CM) whose endomorphism ring is $F_0$-rational.
\begin{theorem}\label{THEOREMBASECHANGEOVERF}
Assume that GRH is true, and let $F_0$ be a number field with no rationally defined CM.
Then there exists an effectively computable constant $B:=B(F_0)\in\Z^+$ for which the following holds: for any finite extension $L/F_0$ whose degree $[L:F_0]$ is coprime to $B$, one has for all elliptic curves $E$ defined over $F_0$ that 
\[
E(L)[\etors]=E(F_0)[\etors].
\]
\end{theorem}
As we will see, the proof of Theorem \ref{THEOREMBASECHANGEOVERF} reduces to determining a finite set $A_{F_0}$ of prime numbers such that for any elliptic curve $E_{/F_0}$ and for any prime order torsion point $R\in E$, the degree $[F_0(R):F_0]$ of the field of definition of $R$ equals $1$ or is divisible by some prime in $A_{F_0}$; in fact, we will show that $[F_0(R):F_0]$ is always even when $\ell$ is sufficiently large.

For our analysis of torsion point degrees, we will prove Theorem \ref{THEOREMSECONDMAIN}, a result on relative uniform bounds for mod-$\ell$ Galois representations of elliptic curves over number fields. The proof will utilize Serre's work on the image of inertia under mod-$\ell$ Galois representations \cite{Ser72} (see also Theorems \ref{DicksonSerreClassification} and \ref{TheoremImageOfInertia}).

Before we state Theorem \ref{THEOREMSECONDMAIN}, let us introduce  several subgroups of $\GL_2(\Z/\ell\Z)$; they are defined in Section \ref{SectionOrbitsGalReps}. Fix an algebraic closure $\oQ$ of $\Q$, and let $G_{F_0}:=\Gal(\oQ/F_0)$ denote the absolute Galois group of $F_0$.
Given an elliptic curve $E_{/F_0}$, for each integer $N\in\Z^+$ let us denote the mod-$N$ Galois representation of $E$ by
\[
\rho_{E,N}\colon G_{F_0}\rightarrow \Aut(E[N]).
\]
When working with a fixed basis $\lbrace P,Q\rbrace$ of $E[N]$, we have an explicit representation $\rho_{E,N, P,Q}\colon G_{F_0}\rightarrow \GL_2(\Z/N\Z)$. By abuse of notation, we will let $G$ denote both images $\rho_{E,N}(G_{F_0})$ and $\rho_{E,N,P,Q}(G_{F_0})$, specifying the basis only when necessary. When $N=\ell$ is prime, we will use $C_s(\ell)$ and $C_{ns}(\ell)$ to denote the split and non-split Cartan subgroups of $\GL_2(\Z/\ell\Z)$, respectively, and $N_s(\ell)$ and $N_{ns}(\ell)$ their respective normalizers.  
We will let $\cD$ denote the semi-Cartan subgroup of $\GL_2(\Z/\ell\Z)$, and $G(\ell)$ the unique subgroup of $N_{ns}(\ell)$ with $G(\ell)\not\subseteq C_{ns}(\ell)$ and $[C_{ns}(\ell):G(\ell)\cap C_{ns}(\ell)]=3$. In the following, we will let $S_{F_0}$ denote the finite set of primes from \cite[Theorem 1]{LV14}, see Remark \ref{RemarkLVHypothesis}.
\begin{theorem}\label{THEOREMSECONDMAIN}
Assume that GRH is true, and let $F_0$ be a number field with no rationally defined CM. Then for all primes $\ell\gg_{F_0}0$, one has that for any elliptic curve $E_{/F_0}$, its mod-$\ell$ representation $G:=\rho_{E,\ell}(G_{F_0})$ is either surjective, or is contained in $N_s(\ell)$ or $N_{ns}(\ell)$ up to conjugacy. 
\begin{enumerate}[1.]
\item If $G\subseteq N_s(\ell)$ then $\cD^{e}\subseteq G$ for some $e\in \lbrace 1,2,3,4,6\rbrace$, and if $\ell\neq 37,73$ then the center $Z(\ell)\subseteq G$, and in fact
\[
[N_s(\ell):G]\mid \gcd(\ell-1,e).
\]
\item If $G\subseteq N_{ns}(\ell)$ then $C_{ns}(\ell)^{e}\subseteq G$ for some $e\in \lbrace 1,2,3,4,6\rbrace$, and in fact
\[
[N_{ns}(\ell):G]\mid 6.
\]
\begin{enumerate}[a.]
\item If $\ell\equiv 1\pmod 3$, then $G$ equals $N_{ns}(\ell)$ or $C_{ns}(\ell)$, with $G=N_{ns}(\ell)$ if $F_0$ has a real embedding.
\item If $\ell\equiv 2\pmod 3$,
then $G$ equals $N_{ns}(\ell)$, $C_{ns}(\ell)$, $G(\ell)$ or $C_{ns}(\ell)^3$, with $G=N_{ns}(\ell)$ or $G(\ell)$ if $F_0$ has a real embedding.
\end{enumerate}
\end{enumerate}
\end{theorem}
\begin{remark}\label{RemarkLVHypothesis}
As we will see in our proof of Theorem \ref{THEOREMSECONDMAIN}, by $\ell\gg_{F_0}0$ we can take $\ell\geq \max\lbrace 29, 15[F_0:\Q]+2\rbrace$, $\ell$ unramified in $F_0$ and $\ell\not\in S_{F_0}$, where $S_{F_0}$ is from \cite[Theorem 1]{LV14}.

Let us briefly explain $S_{F_0}$ and the two assumptions from Theorems \ref{THEOREMBASECHANGEOVERF} and \ref{THEOREMSECONDMAIN}: that GRH is true and $F_0$ has no rationally defined CM.
Under these assumptions, Larson and Vaintrob have shown there exists a finite set of primes $S_{F_0}$ such that the prime degree of any $F_0$-rational isogeny of an elliptic curve lies in $ S_{F_0}$ \cite[Theorem 1]{LV14}. They also give an effectively computable upper bound on the primes from $S_{F_0}$ \cite[Theorem 7.9]{LV14}. 
\end{remark}
\begin{remark}
One has additional information about the mod-$\ell$ Galois representation of a non-CM elliptic curve $E$ defined over $\Q$: for any prime $\ell>1.4\times 10^7$, the mod-$\ell$ representation $\rho_{E,\ell}(G_\Q)$ equals either $\GL_2(\Z/\ell\Z)$ or $N_{ns}(\ell)$ up to conjugacy  \cite[Theorem 1.2]{LFL21}.
\end{remark}
It is natural to ask whether the ``rationally defined CM" hypothesis on $F_0$ is necessary. Indeed it is, as is shown by the following theorem  -- we will prove this in Section \ref{SectionCMCase}.
\begin{theorem}\label{TheoremBaseChangeFailsForRationallyDefinedCM}
Let $F_0$ be a number field with rationally defined CM. Then for all integers $B\in\Z^+$ there exists a finite extension $L/ F_0$ whose degree $[L:F_0]$ is coprime to $B$, and a CM elliptic curve $E_{/F_0}$ for which 
\[
E(L)[\emph{tors}]\neq E(F_0)[\emph{tors}].
\]
\end{theorem}
\section{Towards the Proof of Theorem \ref{THEOREMBASECHANGEOVERF}}
\subsection{Definition of $A_{F_0}$ and $B_{F_0}$}
To prove Theorem \ref{THEOREMBASECHANGEOVERF}, we  need to keep track of prime divisors of torsion point degrees of elliptic curves over $F_0$.
Once and for all, fix an algebraic closure $\oQ$; let $F_0/\Q$ be an algebraic extension.  
Let us call a subset $A_{F_0}\subseteq \Z^+$ of prime numbers \textit{$F_0$-admissible} if for all elliptic curves $E_{/F_0}$ and all prime order torsion points $R\in E(\oQ)$, one has that either $[F_0(R):F_0]=1$ or there exists $p\in A_{F_0}$ so that $p\mid [F_0(R):F_0]$. For example, \cite[Theorem 5.8]{GJN20} shows that $A_{\Q}:=\lbrace 2,3,5,7\rbrace$ is $\Q$-admissible. As we will see in the proof of Proposition \ref{TorsionDoesntChange}, if $A_{F_0}$ is any $F_0$-admissible set and $[L:F_0]$ is coprime to all primes $p\in A_{F_0}$, then for all elliptic curves $E_{/F_0}$ one has $E(L)[\ell]=E(F_0)[\ell]$ for all primes $\ell\in\Z^+$.

Next, let us define $R_{F_0}$ as the set of prime divisors of $F_0$-rational torsion subgroups of elliptic curves,
\[
R_{F_0}:=\lbrace p\in \Z^+: \exists E_{/F_0}~\textrm{with}~p\mid \#E(F_0)[\tors]\rbrace.
\]
Equivalently, $R_{F_0}$ is the set of all primes that are the order of some $F_0$-rational torsion point on an elliptic curve.
If $F_0$ is a number field, then one has that $\#R_{F_0}$ is finite and bounded uniformly in the degree $[F_0:\Q]$ -- this is a consequence of Merel's strong uniform boundedness theorem \cite[Corollaire]{Mer96}.

Following $R_{F_0}$, we define the set $B_{F_0}$ as follows,
\[
B_{F_0}:=\bigcup_{p\in R_{F_0}}\lbrace p\rbrace \cup \lbrace \textrm{prime divisors of }p-1\rbrace.
\]
Obviously $R_{F_0}\subseteq B_{F_0}$, and if $\#R_{F_0}<\infty$ then $\# B_{F_0}<\infty$. 
When $F_0=\Q$, Mazur's torsion theorem \cite{Maz76} shows that $R_\Q=B_\Q=\lbrace 2,3,5,7\rbrace$.
As we will see in the proof of Proposition \ref{TorsionDoesntChange}, $B_{F_0}$ controls which primary torsion point orders can appear upon finite degree base change for any elliptic curve $E_{/F_0}$. 

With $F_0$-admissible sets and $B_{F_0}$ defined, we are able to prove the first key step in the proof of Theorem \ref{THEOREMBASECHANGEOVERF}.
\begin{proposition}\label{TorsionDoesntChange}
Let $F_0/\Q$ be an algebraic extension and $A_{F_0}$ an $F_0$-admissible set. Let $L/F_0$ be any finite extension for which all primes $p\in A_{F_0}\cup B_{F_0}$ do not divide $[L:F_0]$. Then for all elliptic curves $E_{/F_0}$ one has that
\[
E(L)[\emph{tors}]=E(F_0)[\emph{tors}].
\]
\end{proposition}
Before we prove Proposition \ref{TorsionDoesntChange}, let us record an important lemma which controls the degrees of finite extensions over which one attains new prime power torsion.
\begin{lemma}\cite[Proposition 4.6]{GJN20}\label{IntTorsionDegree}
Let $F_0/\Q$ be an algebraic extension, $E_{/F_0}$ an elliptic curve and $R\in E_{/F_0}$ a torsion point of prime order $\ell^n>\ell$. Then $[F_0(R):F_0(\ell R)]$ divides $\ell^2$ or $\ell(\ell-1)$.
\end{lemma}
\begin{proof}[Proof of Proposition \ref{TorsionDoesntChange}]
It suffices to show that for all primes $\ell\in\Z^+$ one has
\[
E(L)[\ell^\infty]=E(F_0)[\ell^\infty],
\]
which is equivalent to showing that for each prime $\ell\in\Z^+$ and for all integers $n>0$ one has
\[
E(L)[\ell^n]=E(F_0)[\ell^n].
\]

We proceed via induction on $n$. The case $n=1$ follows from the assumption that $[L:F_0]$ is coprime to all primes $p\in A_{F_0}$. Suppose then that both $n>1$ and the result is true for $k<n$.
Let $R\in E(L)[\ell^{n}]$ be a point of exact order $\ell^{n}$. Then the inductive hypothesis implies that $\ell^{n-1}R$ is an $F_0$-rational point of order $\ell$, whence $\ell\in R_{F_0}$.
The inductive hypothesis also implies $F_0(\ell R)=F_0$, so by Lemma \ref{IntTorsionDegree} we have
\[
[F_0(R):F_0]\mid \ell^2(\ell-1).
\]
By this divisibility, any prime divisor $p\mid [F_0(R):F_0]$ divides $\ell(\ell-1)$, which from $\ell\in R_{F_0}$ implies $p\in B_{F_0}$ -- which would be impossible since $[F_0(R):F_0]\mid [L:F_0]$ and $[L:F_0]$ is coprime to all primes in $B_{F_0}$.
This forces $[F_0(R):F_0]=1$, which shows that $R\in E(F_0)[\ell^n]$. We
conclude by induction that $E(L)[\ell^\infty]=E(F_0)[\ell^\infty]$.  
\end{proof}

The existence of the constant $B:=B(F_0)\in \Z^+$ in Theorem \ref{THEOREMBASECHANGEOVERF} is equivalent to the existence of a finite $F_0$-admissible set $A_{F_0}$. When assuming GRH, the following proposition gives us a class of number fields for which this happens. 
\begin{proposition}\label{AFFinite}
Assume that GRH is true. Then for any number field $F_0$ which has no rationally defined CM, one can make a finite choice of $F_0$-admissible set $A_{F_0}$.
\end{proposition} 
Theorem \ref{THEOREMBASECHANGEOVERF} is an immediate consequence of combining Propositions \ref{TorsionDoesntChange} and \ref{AFFinite}.
Sans including Section \ref{SectionCMCase}, the rest of our paper is devoted to proving Theorem \ref{THEOREMSECONDMAIN} and then using it to prove Proposition \ref{AFFinite}. Our proofs will involve an analysis of the mod-$\ell$ Galois representation of an arbitrary elliptic curve defined over $F_0$. As noted in Remark \ref{RemarkLVHypothesis}, the assumption that both GRH is true and that $F_0$ has no rationally defined CM will rule out the case where an elliptic curve $E_{/F_0}$ has an $F_0$-rational isogeny of uniformly large prime degree. 
\section{Orbits Under Subgroups of $\GL_2(\Z/\ell\Z)$}\label{SectionOrbitsGalReps}
\subsection{Basic facts about mod-$N$ Galois representations of elliptic curves}
Fix an algebraic extension $F_0/\Q$ and an elliptic curve $E_{/F_0}$. For each integer $N\in\Z^+$, one has an action of the absolute Galois group $G_{F_0}:=\Gal(\oQ/F_0)$ on the set $E[N]$ of $N$-torsion points on $E$. This affords us the \textit{mod-$N$ Galois representation}
\[
\rho_{E,N}\colon G_{F_0}\rightarrow\Aut(E[N]).
\]
$E[N]$ is a rank two $\Z/N\Z$-module, so 
fixing a basis $\lbrace P,Q\rbrace$ for $E[N]$, we have an explicit isomorphism $\Aut(E[N])\cong \GL_2(\Z/N\Z)$. In this case, we can write
\[
\rho_{E,N,P,Q}\colon G_{F_0}\rightarrow \GL_2(\Z/N\Z).
\]
The kernel of this action is $\Gal(\oQ/F_0(E[N]))$, whence we have an isomorphism between $\rho_{E,N,P,Q}(G_{F_0})$ and $\Gal(F_0(E[N])/F_0)$.
We will often work with a fixed basis dependent on the context, but abuse notation and write $G:=\rho_{E,N}(G_{F_0}):=\rho_{E,N,P,Q}(G_{F_0})
$.

It is a standard result that the determinant of the image of $\rho_{E,N}$ is the mod-$N$ cyclotomic character
\[
\chi_N\colon G_{F_0}\rightarrow (\Z/N\Z)^\times.
\]
The fixed field of $\chi_N$ is the $N$-cyclotomic field $F_0(\zeta_N)$, where $\zeta_N$ denotes any primitive $N$'th root of 1. From $\det ( \rho_{E,N}(G_{F_0}))= \chi_N(G_{F_0})$ we find that
\begin{equation}\label{ModEllCharSizeDividesSizeOfGalRep}
\#\chi_N(G_{F_0})\mid \# G.
\end{equation}
If no prime divisors of $N$ ramify in $F_0$, then $\chi_N$ is surjective, i.e., $\det(G)=(\Z/N\Z)^\times$.

For a torsion point $R\in E[N]$, we will write $\oo_G(R)$ for the orbit of $R$ under $G$. By the Orbit-Stabilizer Theorem, we find that
\[
\#\oo_G(R)=[G:\Stab_G(R)]=[\Gal(F_0(E[N])/F_0):\Gal(F_0(E[N])/F_0(R))].
\]
In particular, the degree of $R$ over $F_0$ is the size of its orbit under $G$, 
\[
\#\oo_G(R)=[F_0(R):F_0].
\]
Therefore, the problem of determining prime factors of degrees of torsion points $R\in E[N]$ over $F_0$ is equivalent to determining prime factors of the length of orbits of $E[N]$ under $G_{F_0}$.

Let $L/F_0$ be any finite extension with $L\subseteq F_0(E[N])$.
For a basis $\lbrace P,Q\rbrace$ of $E[N]$, we let $G:=\rho_{E,N, P, Q}(G_{F_0})$ and $H:=\rho_{E,N, P,Q}(G_L)$. Since $G_L\subseteq G_{F_0}$, we have $H\subseteq G$.
Observe that
\begin{equation}\label{EqIndexOfGalSubrep}
[G:H]= [L:F_0].
\end{equation}
This follows from both
\[
[G:H]=\frac{[F_0(E[N]):F_0]}{[L(E[N]):L]}
\]
and $L(E[N])=F_0(E[N])$. In particular, for $L:=F_0(R)$ we get by \eqref{EqIndexOfGalSubrep} the equality
\begin{equation}\label{IndexGalRepOrbitSize}
[G:\rho_{E,N}(G_{F_0(R)})]= \#\oo_G(R).
\end{equation}
It's worth noting that the size of any orbit under $G$ is invariant under a change of basis for $E[N]$.  

\subsection{Subgroups of $\GL_2(\Z/\ell\Z)$}
Our proof of Proposition \ref{AFFinite} uses the classification of subgroups of $\GL_2(\Z/\ell\Z)$ seen in \cite{Ser72}, essentially due to Dickson \cite{Dic58}. We will review this classification.

Throughout the following, we let $\ell\geq 5$ be a prime. We define the \textit{split Cartan subgroup mod-$\ell$} as the subgroup of diagonal matrices,
\[
C_s(\ell):=\left\lbrace \begin{bmatrix}
a&0\\
0&d
\end{bmatrix}:a,d\in\F_\ell^\times\right\rbrace.
\]

Fix the least positive integer $\epsilon$ which generates $\F_\ell^\times$. Then we define the \textit{non-split Cartan subgroup mod-$\ell$} as the regular representation of $\F_\ell[\sqrt{\epsilon}]^\times$ acting on itself via multiplication with respect to the basis $\lbrace 1,\sqrt{\epsilon}\rbrace$; this representation is
\[
C_{ns}(\ell)=\left\lbrace \begin{bmatrix}
a&b\epsilon\\
b&a
\end{bmatrix}: (a,d)\neq (0,0)\in \F_\ell\times\F_\ell\right\rbrace.
\]
Note that $C_{ns}(\ell)\cong \F_{\ell^2}^\times$. 
For these two Cartan subgroups, one has their normalizers
\begin{align*}
N_s(\ell)&=C_s(\ell)\cup \begin{bmatrix}
0&1\\
1&0
\end{bmatrix}C_s(\ell)\\
&=\left\lbrace \begin{bmatrix}
a&0\\
0&d
\end{bmatrix}, \begin{bmatrix}
0&d\\
a&0
\end{bmatrix}:a,d\in \F_\ell^\times\right\rbrace
\end{align*}
and
\begin{align*}
N_{ns}(\ell&)=C_{ns}(\ell)\cup \begin{bmatrix}
1&0\\
0&-1
\end{bmatrix}C_{ns}(\ell)\\
&=\left\lbrace 
\begin{bmatrix}
a&b\epsilon\\
b&a
\end{bmatrix},
\begin{bmatrix}
a&b\epsilon\\
-b&-a
\end{bmatrix}: (a,b)\neq (0,0)\in \F_\ell\times\F_\ell
\right\rbrace.
\end{align*}
Clearly, both Cartan subgroups have index two in their respective normalizers.

Let us also define the \textit{Borel subgroup mod-$\ell$} as the subgroup of upper triangular matrices,
\[
B(\ell):=\left\lbrace \begin{bmatrix}
a&b\\
0&d
\end{bmatrix}:a,d\in \F_\ell^\times\right\rbrace.
\]
In the following section, we will also study the \textit{semi-Cartan subgroup mod-$\ell$}
\[
\cD:=\left\lbrace \begin{bmatrix}
a&0\\
0&1
\end{bmatrix}:a\in \F_\ell^\times\right\rbrace,
\]
and later on we will need the mod-$\ell$ subgroup
\[
G(\ell):=\left\langle \begin{bmatrix}
1&0\\
0&-1
\end{bmatrix},C_{ns}(\ell)^3\right\rangle.
\]

In his seminal paper on the adelic open image theorem for elliptic curves, Serre \cite{Ser72} analyzed the various mod-$\ell$ Galois representations of a fixed elliptic curve over a number field. To do this, he used a classification of subgroups of $\GL_2(\Z/\ell\Z)$ \cite[\S 2]{Ser72}. Towards proving Theorem \ref{THEOREMSECONDMAIN}, we will use this classification to determine explicit subgroups of our mod-$\ell$ Galois representation $\rho_{E,\ell}(G_{F_0})$.
\begin{theorem}[Classification of subgroups of $\GL_2(\Z/\ell\Z)$]\label{DicksonSerreClassification}
Let $\ell\geq 5$ be a prime, and let $G$ be a subgroup of $\emph{GL}_2(\Z/\ell\Z)$. 
\begin{enumerate}[1.]
\item If $\ell\mid \# G$, then one of the following holds:
\begin{enumerate}[a.]
\item $G$ contains $\emph{SL}_2(\Z/\ell\Z)$.
\item $G$ is contained in $B(\ell)$, up to conjugacy.
\end{enumerate}
\item If $\ell\nmid \#G$, then one of the following holds:
\begin{enumerate}[a.]
\item $G$ is contained in $N_s(\ell)$ or $N_{ns}(\ell)$, up to conjugacy.
\item The image $\overline{G}$ of $G$ in $\emph{PGL}_2(\Z/\ell\Z):=\emph{GL}_2(\Z/\ell\Z)/\F_\ell^
\times$ is isomorphic to one of the groups $A_4$, $S_4$ or $A_5$.
\end{enumerate}
\end{enumerate}
\end{theorem}
\subsection{Orbits under the normalizer of a Cartan subgroup}
Let $V:=\lbrace e_1,e_2\rbrace$ be the two-dimensional $\F_\ell$-vector space spanned by the unit vectors $e_1:=\begin{bmatrix}
1\\0
\end{bmatrix}$ and $e_2:=\begin{bmatrix}
0\\1
\end{bmatrix}$. Then one has an action of $\GL_2(\Z/\ell\Z)$ on $V$ via left multiplication.
For a subgroup $G\subseteq \GL_2(\Z/\ell\Z)$ and a vector $v\in V^\bullet:=V\ssm \lbrace 0\rbrace$, we will use $\oo_G(v)$ to denote the orbit of $v$ under $G$. We will also use $\langle v\rangle$ to denote the subspace spanned by $v$.

In our applications, $V:=E[\ell]$ is the $\ell$-torsion subgroup of an elliptic curve.
As noted earlier, by the Orbit-Stabilizer Theorem one has for any point $R\in E[\ell]$ that the size of its orbit $\oo_{\rho_{E,\ell}(G_{F_0})}(R)$ is equal to its degree $[F_0(R):F_0]$. In this subsection, we will see how we can exploit this to determine factors of the degree via the classification in Theorem \ref{DicksonSerreClassification}.

The following proposition describes the orbits of $V^\bullet$ under the action of $N_s(\ell)$. 
\begin{proposition}\label{NormalizerSplitOrbits}
$V^\bullet$ has two $N_s(\ell)$-orbits: they are $\oo_{N_s(\ell)}(e_1)= \oo_{N_s(\ell)}(e_2)$ and $\oo_{N_s(\ell)}(e_1+e_2)$, which are of size $2(\ell-1)$ and $(\ell-1)^2$, respectively.
\end{proposition}
\begin{proof}
Since both $\begin{bmatrix}
a&0\\
0&d
\end{bmatrix}\cdot e_1=ae_1$ and $\begin{bmatrix}
0&d\\
a&0
\end{bmatrix}\cdot e_1=ae_2$, we find that $\oo_{N_s(\ell)}(e_1)=\oo_{N_s(\ell)}(e_2)=\lbrace xe_1+ye_2:\textrm{either }x=0~\textrm{or }y=0\rbrace=\langle e_1\rangle\cup\langle e_2\rangle\ssm \lbrace 0\rbrace$. 
On the other hand, $\begin{bmatrix}
a&0\\
0&d
\end{bmatrix}\cdot (e_1+e_2)=\begin{bmatrix}
0&a\\
d&0
\end{bmatrix}\cdot (e_1+e_2)=ae_1+de_2$, and so $\oo_{N_s(\ell)}(e_1+e_2)=\lbrace xe_1+ye_2: xy\neq 0\rbrace=V^\bullet\ssm (\langle e_1\rangle\cup\langle e_2\rangle)$.
\end{proof}
The following proposition describes the orbits of $V^\bullet$ under $C_{ns}(\ell)$, and thus $N_{ns}(\ell)$.
\begin{proposition}\label{NormalizerNonsplitOrbits}
The group $ C_{ns}(\ell)$ acts transitively on $V^\bullet$. In particular, for each $v\in V^\bullet$ we have $\#\oo_{N_{ns}(\ell)}(v)=\#\oo_{C_{ns}(\ell)}(v)=\ell^2-1$.
\end{proposition}
\begin{proof}
For any $v:=xe_1+ye_2\in V^\bullet$, one sees that the matrix $\begin{bmatrix}
x&y\epsilon\\
y&x
\end{bmatrix}\in C_{ns}(\ell)$ takes $e_1$ to $v$. 
\end{proof}
Recall the following basic fact about orbit divisibility for subgroups.
\begin{lemma}\label{LemmaOrbitDivisibility}
Let a finite group $G$ act on a set $X$. Then for any subgroup $H\subseteq G$, one has for all $x\in X$ that
\[
\#\oo_G(x)\mid [G:H]\cdot \#\oo_H(x).
\]
\end{lemma}
\begin{remark}\label{remarkOrbitDivisibilityForNormCartans}
Suppose we have an elliptic curve $E_{/F_0}$ for which $G:=\rho_{E,\ell,P,Q}(G_{F_0})$ is contained in a subgroup $N(\ell)\subseteq \GL_2(\Z/\ell\Z)$. Then by Lemma \ref{LemmaOrbitDivisibility}, for each $R\in E[\ell]^\bullet$ one has 
\begin{equation}\label{EQIndexDivNormCartanTorsion}
\#\oo_{N(\ell)}(R)\mid [N(\ell):G]\cdot[F_0(R):F_0].
\end{equation}
By Propositions \ref{NormalizerSplitOrbits} and \ref{NormalizerNonsplitOrbits}, this can give us divisibility information about the degree $[F_0(R):F_0]$. For example, Proposition \ref{NormalizerNonsplitOrbits} implies that if $N(\ell)=N_{ns}(\ell)$ then
\[
\ell^2-1\mid [N_{ns}(\ell):G]\cdot[F_0(R):F_0].
\]
When $N(\ell)=N_s(\ell)$, using Proposition \ref{NormalizerSplitOrbits}
one has for nonzero $R\in \langle P\rangle\cup\langle Q\rangle$ that
\[
2(\ell-1)\mid [N_s(\ell):G]\cdot [F_0(R):F_0],
\]
and when $R\not\in \langle P\rangle\cup\langle Q\rangle$ one has instead that
\[
(\ell-1)^2\mid [N_s(\ell):G]\cdot [F_0(R):F_0].
\]
To prove Proposition \ref{AFFinite}, we will use this remark to show that when $\ell$ is sufficiently large with respect to $F_0$, one has for all nontrivial $R\in E[\ell]$ that $[F_0(R):F_0]$ is even.
\end{remark}
\section{The Image of Inertia}
In this section, we will describe the image of inertia under our mod-$\ell$ Galois representation when $\ell\gg_{F_0}0$. This will be towards proving Theorem \ref{THEOREMSECONDMAIN}.
\subsection{The image lies in the normalizer of a Cartan subgroup}
In the following subsection, we will show that under sufficient assumptions, for an elliptic curve $E_{/F_0}$ and a prime $\ell\in\Z^+$, the image $\rho_{E,\ell}(G_{F_0})$ either equals $\GL_2(\Z/\ell\Z)$ or is contained in the normalizer of a Cartan subgroup.
\begin{proposition}\label{ImageLiesInNormalizerOfCartan}
Assume that GRH is true, and fix a number field $F_0$ without rationally defined CM. Then for all primes $\ell> 15[F_0:\Q]+1$ unramified in $F_0$ with $\ell\not\in S_{F_0}$ (see Remark \ref{RemarkLVHypothesis}), for any elliptic curve $E_{/F_0}$ the image $\rho_{E,\ell}(G_{F_0})$ of its mod-$\ell$ Galois representation is either equal to $\emph{GL}_2(\Z/\ell\Z)$ or contained in $N_s(\ell)$ or $N_{ns}(\ell)$ up to conjugacy.
\end{proposition}
\begin{proof}
Let $G:=\rho_{E,\ell}(G_{F_0})$.
By Theorem \ref{DicksonSerreClassification}, either $G$ contains $\SL_2(\Z/\ell\Z)$, is contained in $B(\ell)$, $N_s(\ell)$ or $N_{ns}(\ell)$ up to conjugacy, or has a projective image $\overline{G}:=G/\F_\ell^\times\subseteq \textrm{PGL}_2(\Z/\ell\Z)$ isomorphic to either $A_4, S_4$ or $A_5$. 

Since $\ell$ is unramified in $F_0$, it follows that the mod-$\ell$ cyclotomic character $\chi_\ell\colon G_{F_0}\rightarrow \F_\ell^\times$ is surjective. Since both $\chi_\ell(G_{F_0})=\det(G)$ and $\SL_2(\Z/\ell\Z):=\ker(\det\colon\GL_2(\Z/\ell\Z)\rightarrow \F_\ell^\times)$, it follows that $G$ contains $\SL_2(\Z/\ell\Z)$ iff $G=\GL_2(\Z/\ell\Z)$.

Henceforth, let us assume that $G\neq \GL_2(\Z/\ell\Z)$.
If $\overline{G}$ is isomorphic to either $A_4$, $S_4$ or $A_5$, then Etropolski has shown that $\ell\leq 15[F_0:\Q]+1$ \cite[Proposition 2.6]{Etr}.
The case where $G$ is contained in $B(\ell)$ up to conjugacy is equivalent to $E$ having an $F_0$-rational $\ell$-isogeny, which cannot happen since $\ell\not\in S_{F_0}$ \cite[Theorem 1]{LV14}.
We thus conclude that $G$ is contained in the normalizer of a Cartan subgroup up to conjugacy.
\end{proof}
\subsection{The shape of the inertia subgroup}
Let $F_0$ be a number field and $E_{/F_0}$ an elliptic curve. Following \cite{Ser72},
depending on the reduction type of $E$ at a prime $\Cl\subseteq F_0$ lying above $\ell\geq 5$, we will see that the image $\rho_{E,\ell}(I_\Cl)$ of the inertia group $I_{\Cl}\subseteq G_{F_0}$ contains a uniformly large subgroup.

Let us regard $E$ as lying over the local field $K:=(F_0)_\Cl$, the $\Cl$-adic completion of $F_0$ at $\Cl$. Then there exists a finite extension $L/K$ for which $E$ has semistable reduction \cite[\S VII.5]{Sil09}. Assume $L$ is such an extension of minimal degree. If $E_{/L}$ has bad multiplicative reduction, then $[L:K]\leq 2$ by the theory of Tate curves \cite[\S C.14]{Sil09}. If $E_{/L}$ has good reduction, then the ramification index $e(L/K)\in \lbrace 1,2,3,4,6\rbrace$ -- a more precise version of this is given in e.g. \cite[Theorem 6.1]{LR18}. 

The following result is essentially due to Serre \cite{Ser72}, with a slight generalization to account for elliptic curves defined over a finite extension of $\Q_\ell$. Fix a prime $\ell\in\Z^+$ and an algebraic closure $\oQl$ of $\Q_\ell$. Given a local field $K/\Q_\ell$, we let $K^{\textrm{nr}}$ denote the maximal unramified extension of $K$, and $I_K:=\Gal(\oQl/K^{\textrm{nr}})$ the inertia group of $K$. We also let $K_t/K^{\textrm{nr}}$ denote the maximal tamely ramified extension of $K$, and $I_{K,\ell}:=\Gal(\oQl/K_t)$ the wild inertia group; $I_{K,\ell}$ is a pro-$\ell$-group. Additionally, the tame inertia group is defined as $I_{K,t}:=\Gal(K_t/K^{\textrm{nr}})\cong I_K/I_{K,\ell}$, and is a pro-cyclic group. 

The following theorem describes the image of inertia based on the reduction type of $E_{/L}$, with a few extra conditions to simplify the result.
\begin{theorem}\label{TheoremImageOfInertia}
Let $\ell\geq 5$ be a prime and $K/\Q_\ell$ a finite unramified extension. Let $E_{/K}$ be an elliptic curve, and $L/K$ a minimal extension such that $E_{/L}$ is semistable. Then the absolute ramification index $e:=e(L/\Q_\ell)=e(L/K)\in\lbrace 1,2,3,4,6\rbrace$. Assuming that $\ell\nmid \#\rho_{E,\ell}(I_L)$, one also has the following subgroup of $\rho_{E,\ell}(I_L)$, based on the reduction type of $E_{/L}$:
\begin{itemize}
\item Good ordinary reduction: then $\rho_{E,\ell}(I_L)$ contains $\cD^{e}$ up to conjugacy.
\item Good supersingular reduction: then $\rho_{E,\ell}(I_L)= C_{ns}(\ell)^{e}$ up to conjugacy.
\item Bad multiplicative reduction: same as good ordinary.
\end{itemize}
\end{theorem}
\begin{remark}
Assume that GRH is true, and let $E_{/F_0}$ be an elliptic curve over a number field $F_0$ with no rationally defined CM. If $\ell$ is sufficiently large, then by Proposition \ref{ImageLiesInNormalizerOfCartan} we must have $\ell\nmid \#\rho_{E,\ell}(G_{F_0})$ when $\rho_{E,\ell}(G_{F_0})$ isn't surjective. In such a case, the conditions $e((F_0)_\Cl/\Q_\ell)=1$ and $\ell\nmid \#\rho_{E,\ell}(I_L)$ of Theorem \ref{TheoremImageOfInertia} are automatically satisfied for any prime $\Cl\subseteq F_0$ above $\ell$ and appropriate choice of extension $L/(F_0)_\Cl$.
\end{remark}
\begin{remark}
The theorem above is a very mild generalization of \cite[Theorem 3.1]{LR13} in the case $\ell\nmid \#\rho_{E,\ell}(I_L)$ where our elliptic curve is defined over $(F_0)_\Cl$ instead of $\Q_\ell$. In fact, \cite[Theorem 3.1]{LR13} is a generalization of several results of \cite{Ser72}, where Lozano-Robledo does not assume that $e(L/\Q_\ell)=1$ -- this is to allow his elliptic curves $E_{/\Q}$ to have additive reduction. Since Theorem \ref{THEOREMBASECHANGEOVERF} is a uniformity result independent of reduction type, we must also allow $E_{/F_0}$ to have additive reduction.
\end{remark}
\begin{proof}[Proof of Theorem \ref{TheoremImageOfInertia}]

If $E_{/L}$ has bad reduction, then $e(L/K)\mid 2$ follows from the theory of Tate curves. If $E_{/L}$ has good reduction, then one has $e(L/K)\in \lbrace 1,2,3,4,6\rbrace$ by e.g. \cite[Theorem 6.1]{LR18}. We are thus left to determine the image of inertia. Since $I_{L,\ell}$ is a pro-$\ell$-group and $\ell\nmid \#\rho_{E,\ell}(I_L)$, it follows that the action factors through tame inertia $I_{L,t}$; In particular, the results from \cite{Ser72} for the image $\rho_{E,\ell}(I_{L,t})$ will also apply to $\rho_{E,\ell}(I_L)$.

Suppose that $E_{/L}$ has good ordinary or bad multiplicative reduction. Then by \cite[Propositions 11 and 13]{Ser72}, $I_{L,t}$ acts on the semisimplification of $E[\ell]$ via the trivial character and $\theta^{e}_{\ell-1}$, where $\theta_{\ell-1}\colon I_{L,t}\rightarrow \F_\ell^\times$ is a surjective character.\footnote{For $\ell\nmid d$ and a uniformizer $x$ of $L^{\textrm{nr}}$, one has a character $\theta_d$ via the natural isomorphism $\theta_d\colon \Gal(L^{\textrm{nr}}(\sqrt[d]{x})/L^{\textrm{nr}})\xrightarrow{\sim}\mu_{d}\subseteq \overline{L}$, where $\mu_d$ is the group of $d$'th roots of unity. Such characters parametrize the continuous $\overline{\F_\ell}$-valued characters of $I_{L,t}$ \cite[Proposition 5]{Ser72}.} The kernel 
of the reduction map $E[\ell]\rightarrow \widetilde{E}[\ell]$ is an $L$-rational $\ell$-subgroup, and so one can identify $\rho_{E,\ell}(G_L)\subseteq B(\ell)$. With this identification, one has $\rho_{E,\ell}(I_L)\subseteq\left\lbrace \begin{bmatrix}*&*\\0&1\end{bmatrix}\right\rbrace$. Thus, its semisimplification im$\left(\begin{bmatrix}
\theta_{\ell-1}^{e}&0\\
0&1
\end{bmatrix}\right)=\cD^{e}$ is a subgroup of $\rho_{E,\ell}(I_L)$ by e.g. \cite[Lemma 4]{Gen22}.

Suppose then that $E_{/L}$ has good supersingular reduction. Following the proof of \cite[Theorem 3.1]{LR13}, since $\ell\nmid \#\rho_{E,\ell}(I_L)$ one has that
$I_{L,t}$ acts on $E[\ell]$ via a character $\theta_{\ell^2-1}^{e}$, where $\theta_{\ell^2-1}\colon I_{L,t}\rightarrow \F_{\ell^2}^\times$ is surjective \cite[Proposition 10]{Ser72}. Therefore, $\rho_{E,\ell}(I_L)$ is a cyclic subgroup of $\F_{\ell^2}^\times$ with index $\gcd(\ell^2-1,e)$, and so is isomorphic to the $e$'th power of $C_{ns}(\ell)$.
\end{proof}
Next, we will mildly generalize part of \cite[Theorem 3.2]{LR13}. For a diagonal matrix $\gamma:=\begin{bmatrix}
a&0\\0&d
\end{bmatrix}\in C_s(\ell)$, we let $\gamma_f$ denote its ``flip" $\begin{bmatrix}
d&0\\0&a
\end{bmatrix}$. For a subgroup $H\subseteq C_s(\ell)$, we let $H_f:=\lbrace \gamma_f:\gamma\in H\rbrace$.
\begin{proposition}\label{SubgroupsInRepresentationInsideNormalizerCartan}
With assumptions as in Theorem \ref{TheoremImageOfInertia}, assume further that $\ell\geq 17$ with $\ell\neq 23$. Let us set $H:=\rho_{E,\ell}(I_L)$.
\begin{itemize}
\item If $H\subseteq N_s(\ell)$, then $\cD^{e}\subseteq H$ or $\cD^{e}_f\subseteq H$. 
\item If $H\subseteq N_{ns}(\ell)$, then $H=C_{ns}(\ell)^{e}$.
\end{itemize}
\end{proposition}
\begin{proof}
First, let us assume that $H\subseteq N_s(\ell)$. Then by Theorem \ref{TheoremImageOfInertia}, $\cD^e$ is contained in $H$ up to conjugacy -- otherwise we'd have $C_{ns}(\ell)^{e}=H$ up to conjugacy, which is impossible since $\ell+1\nmid 48$. Thus, for some $M=\begin{bmatrix}
a&b\\
c&d
\end{bmatrix}\in \GL_2(\Z/\ell\Z)$ we have $M\cD^e M^{-1}\subseteq H\subseteq N_s(\ell)$. Since $\cD^e$ is cyclic and $\#\cD^e>2$, it must have an element of order at least $3$, and so one can show through calculations that either $a=d=0$ or $b=c=0$, whence we have $M\cD^e M^{-1}=\cD^e\subseteq H$ or $M\cD^e M^{-1}=\cD_f^e\subseteq H$.

Next, let us assume that $H\subseteq N_{ns}(\ell)$. We claim that $\cD^e$ is not contained in $N_{ns}(\ell)$ up to conjugacy. To see this, we note that elements of $\cD^e$ have eigenvalues $a$ and $1$ where $a\in \F_\ell^\times$, whereas elements of $N_{ns}(\ell)$ have Galois-conjugate eigenvalues of the form $a\pm b\sqrt{\epsilon}$ or $\pm \sqrt{a^2-b^2\epsilon}$ where $(a,b)\in \F_\ell^2\ssm \lbrace (0,0)\rbrace$. As noted in the previous paragraph, $\cD^e$ has an element of order at least 3, which then must have eigenvalues $1$ and $a\in \F_\ell^\times\ssm \lbrace \pm 1\rbrace$; in particular, such an element does not have an eigenvalue pair of the form $a\pm b\sqrt{\epsilon}$ or $\pm \sqrt{a^2-b^2\epsilon}$. We thus deduce by Theorem \ref{TheoremImageOfInertia} that $H$ equals $C_{ns}(\ell)^e$ \textit{up to conjugacy.}

We are left to show that $H=C_{ns}(\ell)^e$, not just up to conjugacy. We claim that $H\subseteq C_{ns}(\ell)$; observe that if this were true, then both $H$ and $C_{ns}(\ell)^e$ are subgroups of the cyclic group $C_{ns}(\ell)$ with equal sizes, whence they must be equal. Both by cyclicity of $H$ and $\# H=\# C_{ns}(\ell)^e$, there exists an element in $H$ of order at least $(\ell^2-1)/6$. Since elements of $N_{ns}(\ell)\ssm C_{ns}(\ell)$ have order dividing $2(\ell-1)$, we deduce that $H$ must be generated by an element of $C_{ns}(\ell)$, whence we conclude that $H=C_{ns}(\ell)^e$.
\end{proof}
Before we prove Theorem \ref{THEOREMSECONDMAIN}, let us record the following simple yet useful fact about non-diagonal subgroups of $N_s(\ell)$.
\begin{lemma}\label{LemmaFlip}
Let $G$ be a subgroup of $N_s(\ell)$ that is not contained in $C_s(\ell)$. Then for all $\gamma\in G\cap C_s(\ell)$, one has $\gamma_f\in G$.
\end{lemma}
\begin{proof}
For any matrix $M\in N_s(\ell)\ssm C_s(\ell)$, one checks that $M\begin{bmatrix}
a&0\\
0&d
\end{bmatrix}M^{-1}=\begin{bmatrix}
d&0\\
0&a
\end{bmatrix}$.
\end{proof}
\section{The proofs of Theorems \ref{THEOREMBASECHANGEOVERF} and \ref{THEOREMSECONDMAIN}}
We will first prove Theorem \ref{THEOREMSECONDMAIN}. After doing so, we will also prove Proposition \ref{AFFinite} in the following subsection, thereby proving Theorem \ref{THEOREMBASECHANGEOVERF}. 
\subsection{The proof of Theorem \ref{THEOREMSECONDMAIN}}
\begin{proof}[Proof of Theorem \ref{THEOREMSECONDMAIN}]
Fix a prime $\ell\geq \max\lbrace 29, 15[F_0:\Q]+2\rbrace$ unramified in $F_0$ where $\ell\not\in S_{F_0}$.
Fix an elliptic curve $E_{/F_0}$, and assume that $\rho_{E,\ell}(G_{F_0})$ is not equal to $\GL_2(\Z/\ell\Z)$. Then by Proposition \ref{ImageLiesInNormalizerOfCartan}, we may fix a basis $\lbrace P,Q\rbrace$ of $E[\ell]$ such that $G:=\rho_{E,\ell,P,Q}(G_{F_0})$ is contained in $N_s(\ell)$ or $N_{ns}(\ell)$. Since $\ell\geq 29$, by Proposition \ref{SubgroupsInRepresentationInsideNormalizerCartan} we also know that $G$ contains $\cD^e, \cD^e_f$ or $C_{ns}(\ell)^e$ for some $e\in \lbrace 1,2,3,4,6\rbrace$. 

Fix a prime $\Cl\subseteq F_0$ over $\ell$, and set $K:=(F_0)_\Cl$.
Taking an extension $L/K$ of minimal degree for which $E_{/L}$ is semistable, we know by Theorem \ref{TheoremImageOfInertia}
that $e:=e(L/K)=e(L/\Q_\ell)\in \lbrace 1,2,3,4,6\rbrace$. 

Suppose first that $G\subseteq N_s(\ell)$; assume that $\ell\neq 37,73$. Since $\ell\not\in S_{F_0}$, we must have $G\not\subseteq C_s(\ell)$. 
By Proposition \ref{SubgroupsInRepresentationInsideNormalizerCartan}
we have $\cD^e\subseteq \rho_{E,\ell}(I_L)$, without loss of generality. We claim that $\rho_{E,\ell}(I_K)\subseteq C_s(\ell)$: for the sake of contradiction, suppose this isn't true. Then since $\rho_{E,\ell}(I_K)$ is cyclic,
it must be generated by an element $M$ in $N_s(\ell)\ssm C_s(\ell)$; such an element has order dividing $2(\ell-1)$. By Lemma \ref{LemmaFlip}, we have both containments $\cD^e\subseteq \rho_{E,\ell}(I_K)$ and $\cD^e_f\subseteq \rho_{E,\ell}(I_K)$. Since $\cD^e$ and $\cD^e_f$ commute with each other, it follows that their product $\cD^e\cD^e_f=C_s(\ell)^e$ is a subgroup of $\rho_{E,\ell}(I_K)$. We thus have that $\#C_s(\ell)^e=(\ell-1)^2/\gcd(\ell-1,e)^2$ divides $\#\rho_{E,\ell}(I_K)$, which is impossible since $\#\rho_{E,\ell}(I_K)\mid 2(\ell-1)$. 
We deduce that $\rho_{E,\ell}(I_K)\subseteq C_s(\ell)$.

The following work will construct a subgroup of $G$ with size $2(\ell-1)^2/\gcd(\ell-1,e)$, which will prove the main index result for the case $G\subseteq N_s(\ell)$. First, we note that $\det\rho_{E,\ell}(I_K)=\chi_\ell(I_K)$, which surjects onto $\F_\ell^\times$ since $K/\Q_\ell$ is unramified. 
Additionally, for any matrix $\gamma\in C_s(\ell)$ one has $\gamma\gamma_f=\det(\gamma)I$. Therefore, since $\rho_{E,\ell}(I_K)\subseteq G\cap C_s(\ell)$ and $G\not\subseteq C_s(\ell)$, Lemma \ref{LemmaFlip} implies that $Z(\ell)\subseteq G$.

Since $\cD^e$ is a subgroup of $G$, it follows that $\cD^e\cD^e_f=C_s(\ell)^e$ is also a subgroup, and has size $(\ell-1)^2/\gcd(\ell-1,e)^2$. In fact, since $C_s(\ell)^e\cap Z(\ell)=Z(\ell)^e$ with size $(\ell-1)/\gcd(\ell-1,e)$, we find that $Z(\ell)C_s(\ell)^e$ is a subgroup of $G$ with size 
\[
\#Z(\ell)C_s(\ell)^e=\#Z(\ell)\cdot \#C_s(\ell)^e\cdot \frac{1}{\#C_s(\ell)^e\cap Z(\ell)}=\frac{(\ell-1)^2}{\gcd(\ell-1,e)}.
\]
Next, we fix an element $N\in G\ssm C_s(\ell)$; writing $N=\begin{bmatrix}
0&b\\
c&0
\end{bmatrix}$, one has $N^2=bcI=-\det(N)I\in Z(\ell)$. As shown in the proof of Lemma \ref{LemmaFlip}, one has for all $\gamma\in C_s(\ell)$ that $N\gamma N^{-1}=\gamma_f$, whence we have $N(Z(\ell)C_s(\ell)^e)N^{-1}=Z(\ell)C_s(\ell)^e$. In particular, $\langle N\rangle Z(\ell)C_s(\ell)^e$ is a subgroup of $G$, whose size is $2|-\det(N)|\cdot \frac{(\ell-1)^2}{\gcd(\ell-1,e)}\cdot \frac{1}{\#\langle N\rangle \cap Z(\ell)C_s(\ell)^e}$. One checks that $\langle N\rangle \cap Z(\ell)C_s(\ell)^e=\langle -\det(N)I\rangle$, from which we deduce that
\[
\#\langle N\rangle Z(\ell)C_s(\ell)^e=\frac{2(\ell-1)^2}{\gcd(\ell-1,e)}.
\] 
From the containments $\langle N\rangle Z(\ell)C_s(\ell)^e\subseteq G\subseteq N_s(\ell)$, we compare indices and conclude that
\[
[N_s(\ell):G]\mid \gcd(\ell-1,e).
\]

Suppose next that $G\subseteq N_{ns}(\ell)$.
This argument will follow a proof of Le Fourn and Lemos \cite[Proposition 1.4]{LFL21}, which itself follows a preprint of Zywina \cite[Proposition 1.13]{Zyw}.
By Proposition \ref{SubgroupsInRepresentationInsideNormalizerCartan} we have $C_{ns}(\ell)^e= \rho_{E,\ell}(I_L)$; the final paragraph of that proof can also be used to show that $\rho_{E,\ell}(I_K)\subseteq C_{ns}(\ell)$.

Since $\rho_{E,\ell}(I_L)$ has index $\gcd(\ell^2-1,e)$ in $C_{ns}(\ell)$, the index of $\rho_{E,\ell}(I_K)$ in $C_{ns}(\ell)$ divides $\gcd(\ell^2-1,e)$. Since $K/\Q_\ell$ is unramified, it follows that $\det(\rho_{E,\ell}(I_K))=\F_\ell^\times$. This forces the index $[C_{ns}(\ell):\rho_{E,\ell}(I_K)]$ to be odd, since otherwise $\rho_{E,\ell}(I_K)$ is a subgroup of the square elements of $C_{ns}(\ell)$, which contradicts surjectivity of the determinant of $\rho_{E,\ell}(I_K)$. We conclude that $[C_{ns}(\ell):\rho_{E,\ell}(I_K)]\mid 3$. 

Let us set $C(G):=G\cap C_{ns}(\ell)$. By the containment $\rho_{E,\ell}(I_K)\subseteq C(G)$, one has $[C_{ns}(\ell):C(G)]\mid 3$; in particular, $C(G)=C_{ns}(\ell)$ or $C_{ns}(\ell)^3$.  Since $\det\colon  \rho_{E,\ell}(I_K)\rightarrow \F_\ell^\times$ is surjective and $\rho_{E,\ell}(I_K)$ is contained in $C(G)$, we have that $\det C(G)=\F_\ell^\times$.

We break our analysis into two cases:
\begin{enumerate}[1.]
\item $\ell\equiv 1\pmod 3$. It follows that $C(G)=C_{ns}(\ell)$, since otherwise $\det C_{ns}(\ell)^3=\F_\ell^\times$, which is impossible since $\# (\F_\ell^\times)^3<\ell-1$. In particular, we have $C_{ns}(\ell)\subseteq G$, and so $G=C_{ns}(\ell)$ or $N_{ns}(\ell)$.
\item $\ell\equiv 2\pmod 3$. If $C(G)=C_{ns}(\ell)$, then $G=C_{ns}(\ell)$ or $N_{ns}(\ell)$. Suppose then that $C(G)=C_{ns}(\ell)^3$.
One checks that $\overline{N_{ns}(\ell)}:=N_{ns}(\ell)/C_{ns}(\ell)^3$ is isomorphic to the dihedral group $D_3$ of order 6, and the quotient group $\overline{G}:=G/C_{ns}(\ell)^3$ has index $I:=[\overline{N_{ns}(\ell)}:\overline{G}]\mid 6$; note that $I=[N_{ns}(\ell):G]$. Keeping in mind that $C(G)=G\cap C_{ns}(\ell)=C_{ns}(\ell)^3$, we have several cases for the index:
\begin{enumerate}[a.]
\item $I=1$: then $G=N_{ns}(\ell)$, which is impossible since $G\cap C_{ns}(\ell)=C_{ns}(\ell)^3$.
\item $I=2$: then $[N_{ns}(\ell):G]=2$. 
Thus we have 
\[
6=[N_{ns}(\ell):C_{ns}(\ell)^3]=2[G:G\cap C_{ns}(\ell)],
\]
and so $[G:G\cap C_{ns}(\ell)]=3$. This is impossible since $[G:G\cap C_{ns}(\ell)]\leq [N_{ns}(\ell):C_{ns}(\ell)]=2$.
\item $I=3$: then the even order element $\gamma:=\begin{bmatrix}
1&0\\
0&-1
\end{bmatrix}$ must lie in $G$ (recall that $N_{ns}(\ell)=\langle \gamma, C_{ns}(\ell)\rangle$). Therefore, since the subgroup $G(\ell):=\langle \gamma,C_{ns}(\ell)^3\rangle$ of $G$ has the property that $[N_{ns}(\ell):G(\ell)]=[N_{ns}(\ell):G]$, we deduce that $G=G(\ell)$.
\item $I=6$: then $[N_{ns}(\ell):C_{ns}(\ell)^3]=[N_{ns}(\ell):G]$, and so $G=C_{ns}(\ell)^3$.
\end{enumerate}
We conclude that when $\ell\equiv 2\pmod 3$ one has that $G$ equals either $N_{ns}(\ell)$, $C_{ns}(\ell)$, $G(\ell)$ or $C_{ns}(\ell)^3$, which are of index $1,2,3$ and $6$ in $N_{ns}(\ell)$, respectively.
\end{enumerate}
Finally, we conclude our proof by noting that if $F_0$ has a real embedding, then it must have an element of trace $0$ and determinant $-1$, from which we get $G\not\subseteq C_{ns}(\ell)$.
\end{proof}
\subsection{The proof of Proposition \ref{AFFinite}}
Recall that a set of primes $A_{F_0}$ is $F_0$-admissible if for any elliptic curve $E_{/F_0}$, any prime $\ell\in\Z^+$ and any point $R\in E[\ell]$, if $R\not\in E(F_0)$ then there exists a prime $p\in A_{F_0}$ for which 
\[
p\mid \#\oo_{\rho_{E,\ell}(G_{F_0})}(R)=[F_0(R):F_0].
\]
Assume that GRH is true, and that $F_0$ has no rationally defined CM. We will now prove Proposition \ref{AFFinite}, which is that one can pick a finite choice of $A_{F_0}$.
\begin{proof}[Proof of Proposition \ref{AFFinite}]
Fix an elliptic curve $E_{/F_0}$; below, for each prime $\ell\in\Z^+$ we will write $G:=\rho_{E,\ell}(G_{F_0})$. First, we observe that we may exclude any finite amount of primes $\ell$ from our analysis. This is because for any point $R\in E[\ell]^\bullet$, the size of its orbit $\oo_{G}(R)$ equals the index of the stabilizer of $R$ in $G$, and this index divides $\#\GL_2(\Z/\ell\Z)=\ell(\ell-1)^2(\ell+1)$. To this end, let us start by adding to $A_{F_0}:=\emptyset$ the set of prime divisors of $\#\GL_2(\Z/\ell\Z)$, where we range over $\ell$ ramified in $F_0$, $\ell< \max\lbrace 29, 15[F_0:\Q]+2\rbrace$, $\ell=37,73$ and $\ell\in S_{F_0}$. 

Fix a prime $\ell\not\in A_{F_0}$. We will show that every order $\ell$  point of $E$ has even degree over $F_0$, which will show that $A_{F_0}$ is $F_0$-admissible.
If $G=\GL_2(\Z/\ell\Z)$, then $G$ acts transitively on $E[\ell]^\bullet$ and thus the degree of any point in $E[\ell]^\bullet$ over $F_0$ is $\ell^2-1$, which is a multiple of 8.
To this end, let us assume that $G\neq \GL_2(\Z/\ell\Z)$; then Proposition \ref{ImageLiesInNormalizerOfCartan} implies that $G$ lies in the normalizer of a Cartan subgroup up to conjugacy. Let us fix a basis $\lbrace P,Q\rbrace$ of $E[\ell]$ for which $G:=\rho_{E,\ell,P,Q}(G_{F_0})$ is contained in $N_s(\ell)$ or $N_{ns}(\ell)$.

\textbf{Suppose $G$ is contained in $N_s(\ell)$.}
Since $\ell\not\in S_{F_0}$, we must have $ G\subseteq N_s(\ell)\ssm C_s(\ell)$.
Fix a point $R\in E[\ell]^\bullet$ which is not $F_0$-rational, and set $H:=\rho_{E,\ell,P,Q}(G_{F_0(R)})$. Then by \eqref{EqIndexOfGalSubrep} we have
\[
[G:H]=[F_0(R):F_0].
\] 
Since $H$ acts trivially on the subgroup $\langle R\rangle$, \cite[Lemma 6.6]{LR13} applies and we land in one of the following 3 cases:
\begin{enumerate}[i.]
\item $H$ is contained in $\cD$;
\item $H$ is contained in $\cD_f$;
\item $H$ equals $ \left\lbrace I, \begin{bmatrix}0&b\\ b^{-1}&0\end{bmatrix}\right\rbrace$ for some $b\in \F_\ell^\times$.
\end{enumerate}
By Theorem \ref{THEOREMSECONDMAIN} we have that $Z(\ell)\subseteq G$. Thus, since $Z(\ell)\cap H=1$ in all three cases, we find that $Z(\ell)H$ is a subgroup of $G$ of size $(\ell-1)\#H$, and so $\ell-1\mid [G:H]$, i.e., $\ell-1\mid [F_0(R):F_0]$, whence we find that $[F_0(R):F_0]$ is even.

\textbf{Suppose $G$ is contained in $N_{ns}(\ell)$.}
Then combining Equation \eqref{EQIndexDivNormCartanTorsion} with Theorem \ref{THEOREMSECONDMAIN}, we have for all points $R\in E[\ell]^\bullet$ that 
\[
\ell^2-1\mid 6\cdot [F_0(R):F_0].
\]
Since $\ell^2\equiv 1\pmod 8$, we deduce that $4\mid [F_0(R):F_0]$.
We thus conclude that our finite set $A_{F_0}$ is $F_0$-admissible.
\end{proof}
\section{The Case Where CM is Rationally Defined}\label{SectionCMCase}
We conclude this paper by proving Theorem \ref{TheoremBaseChangeFailsForRationallyDefinedCM}, which says that the conclusion of Theorem \ref{THEOREMBASECHANGEOVERF} fails when we allow $F_0$ to have rationally defined CM. In particular, we will show in this situation that no $F_0$-admissible set $A_{F_0}$ is finite.
To construct the appropriate counterexamples, we will use a mild generalization of Dirichlet's theorem on primes in arithmetic progressions.
\begin{lemma}[Primes in multiple arithmetic progressions]\label{DirichletTheoremMulti}
Let $M_1,\ldots, M_r$ be pairwise coprime positive integers. For each $1\leq i\leq r$, fix an integer $a_i$ coprime to $M_i$. Then there are infinitely many primes $\ell\in\Z^+$ such that for each $1\leq i\leq r$ one has $\ell\equiv a_i\pmod {M_i}$.
\end{lemma}
\begin{proof}
By the Chinese Remainder Theorem, there exists an integer $x\in \Z$ such that for each $1\leq i\leq r$ we have
\[
x=a_i+M_ik_i
\]
for some $k_i\in \Z$. Since each $\gcd(a_i,M_i)=1$, it follows that $x$ is coprime to the product $M_1M_2\cdots M_r$.

Dirichlet's theorem on primes in arithmetic progressions tells us there are infinitely many primes in the congruency class $\lbrace x+M_1M_2\cdots M_rk:k\in\Z\rbrace$. Let $\ell\in\Z^+$ be any such prime; then we can write 
\[
\ell=x+M_1M_2\cdots M_rk
\]
for some $k\in\Z$.
Thus for each $1\leq i\leq r$ one has
\[
\ell=a_i+M_i(k_i+M_1M_2\cdots M_{i-1}M_{i+1}\cdots M_rk)
\]
from which it follows that $\ell\equiv a_i\pmod {M_i}$.
\end{proof}
As we will see, our proof of Theorem \ref{TheoremBaseChangeFailsForRationallyDefinedCM} requires the existence of an arbitrarily large prime which both splits in at least one imaginary quadratic order in $F_0$ and satisfies certain coprimality conditions.
\begin{lemma}\label{SplittingOandCoprimality}
Let $K$ be an imaginary quadratic field and $\oo\subseteq K$ an order. Then for any integer $B>0$ there are infinitely many primes $\ell$ which split in $\oo$ and, depending on the fundamental discriminant $\Delta_K$, will satisfy one of the following:
\begin{enumerate}[1.]
\item If $\Delta_K=-3$, then $\frac{\ell-1}{6}$ is coprime to $B$.
\item If $\Delta_K=-4$, then $\frac{\ell-1}{4}$ is coprime to $B$.
\item If $\Delta_K<-4$, then $\frac{\ell-1}{2}$ is coprime to $B$.
\end{enumerate} 
\end{lemma}
\begin{proof}
Throughout this proof, we will assume that $B$ is squarefree and $B>2$.
Let us write $K=\Q(\sqrt{-d})$ where $d\in\Z^+$ is squarefree. Given an integer $n\in\Z^+$, let us also write $\textrm{odd}(n)$ for the odd part of $n$.

First, let us suppose that $d=1$, i.e., $\oo=\Z[fi]$ where $f:=[\Z[i]:\oo]$. Then by Lemma \ref{DirichletTheoremMulti} there are infinitely many primes $\ell\in\Z^+$ with $\ell\nmid f$ for which $\ell\equiv -3\pmod {8}$ and $\ell\equiv -1\pmod {\odd(B)}$. The first congruence shows that $\frac{\ell-1}{4}$ is odd, and since $\ell\equiv 1\pmod 4$ we have that $\ell$ splits in $\oo$. From the second congruence, we get that $\ell-1\equiv -2\not\equiv 0\pmod {\odd(B)}$, whence $\odd(B)$ is also coprime to $\frac{\ell-1}{4}$. We conclude that both the splitting and coprimality conditions hold.

Next we suppose that $d=3$, i.e., $\oo=\Z[f\zeta_3]$ where $\zeta_3$ is a primitive cube root of unity and $f:=[\Z[\zeta_3]:\oo]$. Let us write $\odd(B)=3^eB'$ where $\gcd(6,B')=1$ and $0\leq e\leq 1$. Then by Lemma \ref{DirichletTheoremMulti} there are infinitely many primes $\ell\in\Z^+$ with $\ell\nmid f$ for which $\ell\equiv -5\pmod {36}$ and $\ell\equiv -1\pmod { B'}$. One similarly checks that $\frac{\ell-1}{6}$ is coprime to $B$, and since $\ell\equiv 1\pmod 3$ it follows that $\ell$ splits in $\oo$. Thus both desired conditions on $\ell$ hold.

In what follows, we suppose that $d\neq 1, 3$; we will check several cases for $d$. Recall that for any prime $\ell\in\Z^+$, $\ell$ splits in $\oo$ iff both $\ell\nmid f$ and the Legendre symbol
\[
\left(\frac{-d}{\ell}\right)=1.
\]
Let us write
\[
d= 2^{e}\cdot \prod_{i=1}^rp_i\cdot \prod_{j=1}^s q_j,
\]
where $0\leq e\leq 1$ and the primes $p_i\equiv 1\pmod 4$ and $q_j\equiv 3\pmod 4$ (allowing $r=0$ or $s=0$). Then given a prime $\ell\in\Z^+$, the Legendre symbol factorizes as
\[
\left(\frac{-d}{\ell}\right)=\left(\frac{-1}{\ell}\right)\cdot \left(\frac{2}{\ell}\right)^e\cdot \prod_{i=1}^r\left(\frac{p_i}{\ell}\right)\cdot\prod_{j=1}^s\left(\frac{q_j}{\ell}\right).
\]
In particular, if $\ell\equiv -1\pmod 4$ then quadratic reciprocity implies
\begin{equation}\label{EqnLegendreSymbol}
\left(\frac{-d}{\ell}\right)=(-1)^{s+1}\cdot\left(\frac{2}{\ell}\right)^e\cdot \prod_{i=1}^r\left(\frac{\ell}{p_i}\right)\cdot\prod_{j=1}^s\left(\frac{\ell}{q_j}\right).
\end{equation}
The rest of our proof is divided into three cases.
\begin{enumerate}[1.]
\item $e=1$: by Lemma \ref{DirichletTheoremMulti}, there are infinitely many primes $\ell\in\Z^+$ with $\ell\nmid f$ such that $\ell\equiv 3\pmod 8$ and $\ell\equiv -1\pmod {\odd(\lcm(d,B))}$. The former condition implies $\left(\frac{2}{\ell}\right)^e=-1$, and the latter condition implies that each $\left(\frac{\ell}{p_i}\right)=1$ and $\left(\frac{\ell}{q_j}\right)=-1$, whence we have by \eqref{EqnLegendreSymbol} that $\left(\frac{-d}{\ell}\right)=1$. Furthermore, from $\ell\equiv -1\pmod {\odd(B)}$ and $\ell\equiv -1\pmod 4$, we find that $B$ is coprime to $\frac{\ell-1}{2}$. 
\item $e=0$ and $r\neq 0$: fix $1\leq k\leq r$ with $p_k\equiv 1 \pmod 4$, and fix a nonsquare $\alpha$ mod $p_k$. Then by Lemma 15, there are infinitely many primes $\ell\in\Z^+$ with $\ell\nmid f$ such that $\ell\equiv \alpha\pmod {p_k}$ and $\ell\equiv -1\pmod {4\cdot \frac{\odd(\lcm(d,B))}{p_k}}$. We check by \eqref{EqnLegendreSymbol} that $\left(\frac{-d}{\ell}\right)=1$, and that $B$ is coprime to $\frac{\ell-1}{2}$. 
\item $e=0$ and $r=0$: then there exists $1\leq k\leq s$ with $q_k\neq 3$, since otherwise we have $d=3$. Fix such a $k$, and fix a square $\alpha\not\equiv 1\pmod{q_k}$. Then by Lemma \ref{DirichletTheoremMulti}, there are infinitely many primes $\ell\in\Z^+$ with $\ell\nmid f$ such that $\ell\equiv \alpha\pmod {q_k}$ and $\ell\equiv  -1\pmod{4\cdot \frac{\odd(\lcm(d,B))}{q_k}}$. Thus by \eqref{EqnLegendreSymbol} we have that both $\ell$ splits in $\oo$ and $B$ is coprime to $\frac{\ell-1}{2}$. \qedhere
\end{enumerate}
\end{proof}
\begin{proof}[Proof of Theorem \ref{TheoremBaseChangeFailsForRationallyDefinedCM}]
We will first prove the case where $F_0:=K_\oo$ is a ring class field of an order $\oo$ from an imaginary quadratic field $K$.

Fix an integer $B\in\Z^+$. First, let us assume that the fundamental discriminant $\Delta_K\neq -3,-4$. Then by Lemma \ref{SplittingOandCoprimality}, there exists a prime $\ell\geq 7$ for which both $\ell$ splits in $\oo$ and $\frac{\ell-1}{2}$ is coprime to $B$.
By \cite[Corollary 1.8]{BC20}, there exists an $\oo$-CM elliptic curve $E_{/K_\oo}$ whose mod-$\ell$ Galois representation image $\rho_{E,\ell}(G_{K_\oo})$ equals $C_s(\ell)$ with respect to some basis $\lbrace P,Q\rbrace$ of $E[\ell]$. It follows that the $G_{K_\oo}$-stable subgroup $\langle P\rangle$ induces a $K_\oo$-rational $\ell$-isogeny of $E$, and since $\rho_{E,\ell}(G_{K_\oo})=C_s(\ell)$ the degree of $P$ over $K_\oo$ is
\[
[K_\oo(P):K_\oo]=\ell-1.
\]
Let $r\colon G_{K_\oo}\rightarrow \F_\ell^\times$ denote the isogeny character of $\langle P\rangle$, i.e., the action of $G_{K_\oo}$ on $\langle P\rangle$. Then we have
\[
\# r(G_{K_\oo})=[K_\oo(P):K_\oo]=\ell-1.
\]

Let us define a new character $\chi\colon G_{K_\oo}\rightarrow \F_\ell^\times$ via
\[
\chi(\sigma):=r(\sigma)^{\frac{\ell-1}{2}}.
\]
Thus $\chi(\sigma)=1$ if $r(\sigma)$ is a square mod $\ell$, and is $-1$ otherwise. Let $E'$ be the quadratic twist of $E$ by $\chi$; then $E'$ is also defined over $K_\oo$ and has CM by $\oo$. Furthermore, $P$ induces a point $P'\in E'[\ell]^\bullet$ that generates a $K_\oo$-rational $\ell$-isogeny, and whose isogeny character is $r':=\chi
\cdot r$.

We claim that the order $\#r'(G_{K_\oo})=\frac{\ell-1}{2}$. To see this, we first note that for each $\sigma\in G_{K_\oo}$, if the image $r(\sigma)$ has even order, then it must be a non-square modulo $\ell$; this follows from $\ell\equiv 3\pmod 4$. It follows that $r'(\sigma)=-r(\sigma)$ has odd order equal to $\frac{\ell-1}{\gcd(\ell-1, \frac{\ell-1}{2}+k)}$, where we've written $\langle g\rangle=\F_\ell^\times$ and $r(\sigma)=g^k$. So taking a generator $g:=r(\sigma)$ of $\F_\ell^\times$, it follows that $r'(\sigma)$ has order $\frac{\ell-1}{2}$, whence we deduce that $r'(G_{K_\oo})$ is an index two subgroup of $\F_\ell^\times$. 

By the above work, we conclude that
\[
[K_\oo(P'):K_\oo]=\frac{\ell-1}{2},
\]
whence the degree of $P'$ over $K_\oo$ is coprime to $B$. Since $\frac{\ell-1}{2}>1$, we find that $P'$ is not $K_\oo$-rational.
We conclude that
\[
E'(L)[\ell]\supsetneq E'(K_\oo)[\ell]
\]
where $L:=K_\oo(P')$. In particular, for a set $A_{K_\oo}$ of primes to be $K_\oo$-admissible, for any integer $B\in\Z^+$ it must contain a prime coprime to $B$, which forces $\#A_{K_\oo}=\infty$.

Next, we suppose that $\Delta_K=-4$, and so $\oo\subseteq\Z[i]$. By Lemma \ref{SplittingOandCoprimality}, we can choose a prime $\ell\geq 13$ which splits in $\oo$ -- by its proof, let us choose $\ell\equiv -3\pmod {8}$ -- and for which $\frac{\ell-1}{4}$ is coprime to $B$. Again by \cite[Corollary 1.8]{BC20} there exists an $\oo$-CM elliptic curve $E_{/K_\oo}$ for which
\begin{equation}\label{CMGalRepFullSplit}
\rho_{E,\ell}(G_{K_\oo})=C_s(\ell)
\end{equation}
with respect to some basis $\lbrace P,Q\rbrace$ of $E[\ell]$. 

By \cite[Proposition 2.2]{BP16} there exists a $\Z[i]$-CM elliptic curve $E''_{/K_\oo}$ and a $K_\oo$-rational cyclic $f$-isogeny $E\rightarrow E''$, where $f:=[\Z[i]:\oo]$. Since $\ell\nmid f$, this restricts to a $\F_\ell[G_{K_\oo}]$-module isomorphism $E[\ell]\xrightarrow{\sim}E''[\ell]$, whence we have isomorphic mod-$\ell$ Galois representations, $\rho_{E,\ell}(G_{K_\oo})\cong \rho_{E'',\ell}(G_{K_\oo})$.

So without loss of generality, let us assume that $\oo=\Z[i]$. Then \eqref{CMGalRepFullSplit} tells us there is a $K_\oo$-rational $\ell$-isogeny $\langle P\rangle\subseteq E[\ell]$ for which $[K_\oo(P):K_\oo]=\ell-1$; write its isogeny character as $r\colon G_{K_\oo}\rightarrow \F_\ell^\times$. Let us define the quartic character $\chi:=r^{\frac{\ell-1}{4}}\colon G_{K_\oo}\rightarrow \F_\ell^\times$. Then one obtains a quartic twist $E'_{/K_\oo}$ of $E$ via $\chi$. Corresponding to $P\in E$, one has an $\ell$-torsion point $P'\in E'$ whose isogeny character is $\chi \cdot r$. Since $\ell\equiv -3\pmod {8}$, we find that the image $(\chi\cdot r)(G_{K_\oo})$ is an index 4 subgroup of $\F_\ell^\times$, whence we have 
\[
[K_\oo(P'):K_\oo]=\frac{\ell-1}{4}.
\]
In particular, over the extension $K_\oo(P')$ of $K_\oo$, the elliptic curve $E'$ obtains a new torsion point $P'$ whose nontrivial degree $[K_\oo(P'):K_\oo]=\frac{\ell-1}{4}$ is coprime to $B$. We conclude that $\#A_{K_\oo}=\infty$ for any $K_\oo$-admissible set $A_{K_\oo}$.

A similar analysis shows that when $K=\Q(\sqrt{-3})$ there exists a prime $\ell\geq 31$ with $\ell\equiv -5\pmod {36}$, for which a sextic twist $E'$ of a $\Z[\zeta_3]$-CM elliptic curve $E_{/K_\oo}$ with $\rho_{E,\ell}(G_{K_\oo})=C_s(\ell)$ will have an $\ell$-torsion point $P'$ of degree $\frac{\ell-1}{6}>1$ over $K_\oo$, and for which $\frac{\ell-1}{6}$ is coprime to $B$. This concludes the proof in the case where the base field $F_0$ is a ring class field $K_\oo$ of an imaginary quadratic order.

Finally, let us assume that $F_0$ contains a ring class field $K_\oo$ of an imaginary quadratic field. Then our work above shows that for $B\in\Z^+$, there exists a prime $\ell\gg_{K_\oo, B} 0$, an $\oo$-CM elliptic curve $E_{/K_\oo}$ and a torsion point $P\in E[\ell]$ such that $E(K_\oo(P))[\ell]\neq E(K_\oo)[\ell]$ and the degree $[K_\oo(P):K_\oo]$ is coprime to $B$. Since $\ell$ is allowed to be arbitrarily large with respect to $F_0$, we may assume that $[F_0(P):F_0]>1$; for example, by Merel's Theorem we can take $\ell>\max\lbrace 7,[F_0:\Q]^{3[F_0:\Q]^2}\rbrace$ \cite[Th\'{e}or\`{e}me]{Mer96}. 
Furthermore, since $K_\oo(P)/K_\oo$ is Galois we find that $[F_0(P):F_0]\mid [K_\oo(P):K_\oo]$,
and so it follows that the degree of $P$ over $F_0$ is also coprime to $B$. We deduce that $E(F_0(P))[\ell]\neq E(F_0)[\ell]$, and so we conclude that any $F_0$-admissible set $A_{F_0}$ must have $\#A_{F_0}=\infty$.
\end{proof}
\subsection{Acknowledgments}
The author thanks the referee for their comments and suggestions.

\end{document}